\documentclass{amsart}

\usepackage{amssymb,amsmath,amsthm,latexsym}
\usepackage{booktabs}
\usepackage[all]{xy}

\newtheorem{theorem}{Theorem}[section]
\newtheorem{definition}[theorem]{Definition}
\newtheorem{lemma}[theorem]{Lemma}
\newtheorem{corollary}[theorem]{Corollary}

\newtheorem{proposition}{Proposition}[section]
\newtheorem{remarks}[theorem]{Remarks}

\begin{document}

\title[On Operator-valued Semicircular Random Variables]
{On Operator-valued Semicircular Random Variables}

\maketitle

\begin{center}
\author {Mohsen Soltanifar\footnote{Mohsen Soltanifar \newline Department of Mathematics and Statistics, University of Saskatchewan, 106 Wiggins Road,  Saskatoon, S7N 5E6, Canada \newline e-mail: mohsen.soltanifar@usask.ca}}
\end{center}

\begin{abstract}

In this paper, we discuss some special properties of operator-valued semicircular random variables including representation of the Cauchy transform of a compactly supported probability measure in terms of their operator-valued Cauchy transforms and existence of nonzero discrete part of their associated distributions.

\end{abstract}

																												 



\section{Introduction}

In his 1995 paper on operator-valued free probability theory, which became one of the main pillars of free probability , Voiculescu ~\cite{DV} introduced the operator-valued free central limit theorem and operator-valued semicircular random variables as operator-valued free analogous of the classical central limit theorem and normal random variables, respectively. Similar to normal random variables in classical probability theory, operator-valued semicircular random variables play key roles in many areas of operator-valued free probability and related fields. Following Wigner ~\cite{W} and Voiculescu ~\cite{DV0}, Shlyakhtenko showed in ~\cite{SHL} that operator-valued semicircular random variables also play an important role in random matrix theory, as asymptotic distributional limits of Gaussian band matrices. Later this result found several applications. We shall mention here the paper of Rashidi Far and others ~\cite{RR - OT - BW - SR} which shows the importance of matrix-valued semicircular random variables in communication theory.\\

This paper deals with exploring some other aspects of operator-valued semicircular random variables. Our investigation originates from a question of Speicher on the existence of a discrete part in the spectrum of a matrix-valued semicircular random variable. Following well-established methods in noncommutative probability, we approach this problem through the study of distributions of such a random variable with respect to convenient positive linear functionals. In our work we were led to question roughly which probability measures can occur as scalar-valued distributions of operator-valued semicircular random variables. It turns out that  the operator-valued semicircular random variables have a certain universal property: For given arbitrary compactly supported probability measure, its associated Cauchy transform is the composition of an extremal state and the operator-valued Cauchy transform of an infinite dimensional matrix-valued semicircular random variable (Theorem \ref{th: 12}). Moreover, our proof gives a constructive method to find this semicircular random variable based on the continued fraction representation of the Cauchy transform of the given probability measure. Our second result directly concerns the discrete part of the distribution of a matrix-valued semicircular random variable with respect to the linear functional induced by the composition of the expectation with the matrix trace. It was noted by Speicher  (private communications with Belinschi) that an $M_{2}(\mathbb{C})-$valued semicircular random variable with variance
$\eta\left(\begin{array}{cc}z & v\\
y & w \end{array}\right)
=\left(\begin{array}{cc}
0 & 1\\
0 & 0
\end{array}
\right)
\left(\begin{array}{cc}
z & v\\
y & w
\end{array}
\right)
\left(\begin{array}{cc}
0 & 0\\
1 & 0
\end{array}
\right)
$  has a purely discrete distribution equal to $\mu=\frac{1}{4}\delta_{+1}+\frac{1}{2}\delta_{0}+\frac{1}{4}\delta_{-1}.$ Motivated by this observation and the proof of the Proposition \ref{pro:11}, we were led to show in Theorem \ref{th:9} that any matrix-valued semicircular random variable with nilpotent variance has an atom at zero. Finally, by combining the methods used in the proofs of the previous results, it is shown that the semicircular distributions of finite dimensional matrix-valued semicircular random variables can cover finitely supported probability measures in a special sense. This paper is divided into three sections including preliminaries, representation of  the Cauchy transform using semicircular random variables and a discussion of atoms of distributions of matrix-valued semicircular random variables.


\section{Preliminaries}

The reader who has studied the concepts of Cauchy or Cauchy-Stieltjes transform, continued fraction, operator-valued noncommutative random variables, and semicircular distributions is well acquainted with the  following definitions and results. For an essential account of the mentioned concepts, see
~\cite{AH - NO}, ~\cite{JS - JT}, and ~\cite{DV}. Throughout this paper it is assumed that the probability measure $\mu$ on $\mathbb{R}$ is compactly supported.  We begin with a definition:

\begin{definition}
Let $\mu$ be a probability measure on the Borel $\sigma$-algebra of $\mathbb{R}$ . The associated Cauchy transform $G_{\mu}$ to $\mu$ is defined by:
$$
G_{\mu}(\xi)=\int_{\mathbb{R}}\frac{d\mu(t)}{\xi-t}\ \ \ \ Im\xi\neq0.
$$
\end{definition}
Some properties of $G_{\mu}$ are listed in the following proposition. We denote by $\mathbb{C}^{+}=\{\xi\in\mathbb{C}:Im \xi>0\},$ $Supp(\mu)=\mathbb{R}\setminus\bigcup \{U\subseteq\mathbb{R}|U\ open, \mu(U)=0\},$ and  $\Gamma_{\alpha}(r)=\{\xi\in\mathbb{C}^{+}:|Re(\xi)-r|<\alpha Im(\xi)\}\ \ (-\infty<r<\infty).$
\begin{proposition}
\label{pro:2}
Let $G=G_{\mu}$ be the Cauchy transform of a  probability measure $\mu$ on $\mathbb{R}.$ Then:\newline\\
(i)\ $G$ is analytic on $\mathbb{C}\setminus Supp(\mu),$\\
(ii)\ If $\xi\in\mathbb{C}^{+}(\mathbb{C}^{-}),$ then $G(\xi)\in\mathbb{C}^{-}(\mathbb{C}^{+}),$\\
(iii)\ $\lim_{\Gamma_{\alpha}(0)\ni\xi\rightarrow\infty}\xi.G(\xi)=1,$ for any fixed $0<\alpha<\infty.$
\end{proposition}

Note that any function can be identified as a Cauchy transform of a probability measure on $\mathbb{R}$ by possessing the three mentioned properties in the Proposition above ~\cite{JS - JT}. In addition, it is straightforward to check that:
\begin{eqnarray}
\label{eq:3}
\lim_{\Gamma_{\alpha}(r)\ni\xi\rightarrow r}(\xi-r)G_{\mu}(\xi)&=&\mu(\{r\})\ \ (-\infty<r<\infty). \\
\nonumber
\end{eqnarray}

\begin{definition}
A probability measure $\mu$ on $\mathbb{R}$ is said to have finite moment of order $m\geq1$ if $\int_{\mathbb{R}}|t|^{m}d\mu(t)<\infty,$ and in this case the $m^{th}$ moment of $\mu$ is defined by $M_m=\int_{\mathbb{R}}t^{m}d\mu(t).$
\end{definition}

We denote the set of all Borel probability measures on $\mathbb{R}$ having finite moments of all orders by $\mathcal{B}_{fm}(\mathbb{R}).$ It is trivial that any compactly supported probability measure $\mu$ is in this set, and as a corollary of the Carleman's moment test ~\cite[Theorem 1.36]{AH - NO} it is a solution of the determinate moment problem. Next, using the idea of associated Gram-Schmidt orthonormal polynomials to a probability measure $\mu$, one can prove the existence of the so-called Jacobi coefficients of the given probability measure $\mu\in\mathcal{B}_{fm}(\mathbb{R}),$~\cite[Theorem 1.44]{AH - NO}:
\begin{theorem}
Let $\{p_m(t)\}_{m=0}^{\infty}$ be the Gram-Schmidt orthonormal polynomials associated with given $\mu\in\mathcal{B}_{fm}(\mathbb{R}).$ Then there exists a pair of sequences $\{\alpha_m\}_{m=1}^{\infty}\subseteq\mathbb{R}$ and $\{\omega_m\}_{m=1}^{\infty}\subseteq\mathbb{R}^{+}$ uniquely determined by:
\begin{eqnarray}
\label{eq:4}
p_0(t)&=&1,\nonumber\\
p_1(t)&=&t-\alpha_1,\\
tp_m(t)&=&p_{m+1}(t)+\alpha_{m+1}p_m(t)+\omega_{m}p_{m-1}(t)\ m\geq1,
\nonumber
\end{eqnarray}
where in which, if $|Supp(\mu)|=\infty,$ both $\{\alpha_m\}_{m=1}^{\infty},\{\omega_m\}_{m=1}^{\infty}$ are infinite sequences, and if $|Supp(\mu)|=m_0<\infty,$ we have $\{\alpha_m\}_{m=1}^{\infty}=\{\alpha_m\}_{m=1}^{m_0}$ and $\{\omega_m\}_{m=1}^{\infty}=\{\omega_m\}_{m=1}^{m_0-1}$ with $p_{m_0}=0.$
\end{theorem}
Note that as a corollary of the equations \eqref{eq:4}, for the  compactly supported probability measure $\mu\in\mathcal{B}_{fm}(\mathbb{R})$  we have:
\begin{eqnarray}
\label{ineq:5}
\sup_{m}(|\alpha_m|+\omega_m)&\leq&2\sup_{t\in Supp(\mu)}|t|<\infty.\\
\nonumber
\end{eqnarray}
The following theorem gives a continued fraction representation of the associated Cauchy transform of any  probability measure  $\mu\in\mathcal{B}_{fm}(\mathbb{R}),$~\cite[Theorem 1.97]{AH - NO}:
\begin{theorem}
\label{th:8}
Let $\mu\in\mathcal{B}_{fm}(\mathbb{R})$ and $(\{\omega_n\}_{n=1}^{\infty},\{\alpha_n\}_{n=1}^{\infty})$ be its Jacobi coefficients. If $\mu$ is the solution of the determinate moment problem, then the Cauchy transform of it is expanded into a continued fraction
$$
G_{\mu}(\xi)=\frac{1}{\xi-\alpha_1-\displaystyle{\frac{\omega_1}{\xi-\alpha_2-\displaystyle{\frac{\omega_2}{{\xi-\alpha_3-}_{\displaystyle{\ddots}}}}}}}\ \ \ \  Im(\xi)\neq0.
$$
\end{theorem}
Before introducing the distribution of operator-valued semicircular random variables, we remind the following essential definitions and results from the operator-valued noncommutative probability theory:
\begin{definition}
\label{def:6}
(1)\ Let $A$ be a unital $*$-algebra and let $B$ denote a fixed unital $*$-subalgebra of $A$ over $\mathbb{C}$. A linear map $E_{B}:A\rightarrow B$ is called a conditional expectation if it satisfies the following conditions:
\newline \\
(i) $E_{B}(b_1ab_2)=b_1E_{B}(a)b_2$ for all $a\in A,\ b_1,b_2\in B$, and $E_{B}(1)=1,$\\
(ii) $E_{B}(a^{*}a)\geq0,$ for all $a\in A,$
\newline\\
(2)\  A triple $(A,E_{B},B)$ as in part (1) is called a $B$-valued non-commutative probability space. An element $a\in A$ is called a $B$-valued random variable,\\
(3)\ Let $(A,E_{B},B)$ be as in part (2), and $B\subseteq A_i\subseteq A (i\in I)$ be subalgebras. The family $\{A_i\}_{i\in I}$ is called free over $B$ if $$E_{B}(a_{i_1}a_{i_2} \cdots a_{i_n})=0$$
whenever $i_1\neq i_2,i_2\neq i_3, \cdots, i_{n-1} \neq i_n, a_{i_j}\in A_{i_j}$ and $ E_{B}(a_{i_j})=0, (1\leq j\leq n).$
\end{definition}

We call the family  $\{X_i\}_{i\in I}$ of subsets of A (elements $\{a_i\}_{i\in I}$ of A) free if the corresponding family of subalgebras $\{\langle X_i\cup B\rangle\}_{i\in I}$ ($\{\langle\{a_i\}\cup B\rangle\}_{i\in I}$) is free.\\

Given an operator-valued noncommutative probability space $(A,E_{B},B)$ and a $B$-valued random variable $a\in A,$ the associated moments of $a$ are, by definition, the  multilinear functionals $\{m_{n}\}_{n=0}^{\infty}$ defined by:
\begin{eqnarray}
&& m_{n}: B^{n}\rightarrow B,\nonumber\\
&& m_{n}(b_1,b_2,\cdots,b_n)=E_{B}(ab_{1}ab_{2}\cdots ab_{n}a),\nonumber
\nonumber
\end{eqnarray}
in which the quantity $m_0=E_{B}(a)\in B$ is called the first moment, the map $b_1\mapsto E_{B}(ab_{1}a)$ is called the second moment, and in general the map $(b_1,b_2,\cdots,b_n)\mapsto E_{B}(ab_{1}ab_{2}\cdots ab_{n}a)$ is called the $(n+1)^{th}$ moment. Next, let $B$ be a Banach algebra, and $a,a_1,\cdots,a_m,\cdots$ be a sequence of of $B$-valued random variables in $A$ with associated sequences of moments $\{m_{n}\}_{n=0}^{\infty},\{m_{n}^{(1)}\}_{n=0}^{\infty},\cdots,\{m_{n}^{(m)}\}_{n=0}^{\infty},\cdots,$ respectively. We say that the sequence $\{a_{m}\}_{m=1}^{\infty}$ convergence to  $a$ in moments if
$$
\lim_{m\rightarrow\infty}\|m_{n}^{(m)}(b_1,b_2,\cdots,b_n)-m_{n}(b_1,b_2,\cdots,b_n)\|=0,
$$
for all $(b_1,b_2,\cdots,b_n)\in B^{n},$ and $n\geq0.$ \\

Having the same assumptions as above, we recall that any $a\in A$ can be written as $a=Re(a)+i.Im(a)$ where $Re(a)=\frac{a+a^{*}}{2}$ and $Im(a)=\frac{a-a^{*}}{2i}$ are self-adjoint elements. We define $\mathbb{H}^{+}(A)=\{a\in A|Im(a)>0\},$ where $Im(a)>0$ means $Im(a)>\epsilon.1$ for some $\epsilon>0$, and similarly $\mathbb{H}^{+}(B).$ Then the operator-valued Cauchy transform of $a\in A$ is an analytic map $G_{a}$ defined via:
\begin{eqnarray}
&&G_{a}:\mathbb{H}^{+}(B)\rightarrow \mathbb{H}^{-}(B)\nonumber\\
&&G_{a}(b)=E_{B}((b-a)^{-1})\nonumber\\
&&\ \ \ \ \ \ \ \ =\sum_{n=0}^{\infty}b^{-1}E_{B}((ab^{-1})^{n})\ \ b\in \mathbb{H}^{+}(B), \|b^{-1}\|<\|a\|^{-1}.\nonumber
\nonumber
\end{eqnarray}
Next, the operator-valued $R-$transform of $a\in A,$ $R_{a}:B\rightarrow B,$ can be defined by the relation
$$
bG_{a}(b)=1+R_{a}(G_{a}(b)).G_{a}(b)\ \ (b\in B),
$$
~\cite[Theorem 4.9.]{DV}. The following central limit theorem for operator-valued random variables is due to Voiculescu, ~\cite[Theorem 8.4.]{DV}:
\begin{theorem}
(Free Central Limit Theorem) Let $B$ be a Banach algebra and $a_{1},a_{2},\cdots,a_{m},\cdots$ be a  sequence of free $B$-valued random variables in the non-commutative operator valued probability space $(A, E_{B},B)$such that:
\newline\\
(i)$E_{B}(a_m)=0,\ (m\in\mathbb{N}),$\\
(ii) there is a bounded linear map $\eta:B\rightarrow B$ such that $$\lim_{n\rightarrow\infty}\frac{\sum_{m=1}^{n}E_{B}(a_mba_m)}{n}=\eta(b),\ (b\in B),$$
(iii) there are constants $C_k\ (k\in\mathbb{N})$ such that $$\sup_{m\in\mathbb{N}}\| E_{B}(a_mb_1a_m \cdots b_ka_m)\|\leq C_k\| b_1\| \cdots \| b_k\|\ (k\in\mathbb{N}).$$
Then the sequence $S_m=\frac{\sum_{k=1}^{m}a_k}{\sqrt{m}}\ (m\in\mathbb{N})$ converges in moments.
\end{theorem}
The central limit in the above theorem which we shall denote it by $s$ is called a $B$-valued semicircular element in the context of operator-valued noncommutative probability, and, as in classical probability theory it is uniquely determined by its first two moments. Indeed, its $R-$ transform is of the form $$R_{s}(b)=D+\eta(b),\ (b\in B)$$ where $D=E_{B}(s)\in B$ is a self-adjoint element and $\eta:B\rightarrow B$ is a completely positive map given by $\eta(b)=E_{B}(sbs)-E_{B}(s)bE_{B}(s)\ (b\in B).$ Furthermore, for any completely positive map $\eta:B\rightarrow B$ there exists a $B-$valued semicircular random variable $s$ such that  $\eta(b)=E_{B}(sbs)-E_{B}(s)bE_{B}(s)\ (b\in B),$ ~\cite[Theorem 4.3.1.]{RS2}.\\

The following  result of Helton, Rashidi Far, and Speicher, shows that any operator-valued semicircular random variable can be uniquely determined by a functional equation involving only its first two moments, ~\cite{WH - RR - RR}:

\begin{theorem}
\label{th:7}
Let $A$ be a unital $C^{*}$-algebra, $B$ a $C^{*}$-subalgebra of $A$, and $s\in A$ be a self-adjoint $B$-valued semicircular random variable with first moment $D\in B$ and variance $\eta:B\rightarrow B.$ Then its associated operator-valued Cauchy transform $G_{s}:\mathbb{H}^{+}(B)\rightarrow \mathbb{H}^{-}(B)$ is the unique solution of the functional equation
\begin{eqnarray}
\label{eq:8}
b.G_{s}(b)=1+(D+\eta(G_{s}(b))).G_{s}(b),\ \ \ (b\in \mathbb{H}^{+}(B)),\\
\nonumber
\end{eqnarray}
together with asymptotic condition
\begin{eqnarray}
\label{eq:9}
\lim_{b^{-1}\rightarrow 0}b.G_{s}(b)=1.\\
\nonumber
\end{eqnarray}
\end{theorem}
Let $A$ be a unital $C^{*}$-algebra, $B$ be a unital $C^{*}$-subalgebra of $A$, $E_{B}:A\rightarrow B$ be a conditional expectation, and $G_{a}:\mathbb{H}^{+}(B)\rightarrow\mathbb{H}^{-}(B)$ be the operator-valued Cauchy transform of the self-adjoint  random variable $a\in A.$ Let $\Phi:B\rightarrow\mathbb{C}$ be a given state on the $C^{*}$-algebra $B$ (In the case of $B=M_{n}(\mathbb{C})$, for some fixed $n\geq1,$ we can take $\Phi=tr_{n},$ the normalized trace.), and define a map $G\colon\mathbb{C}^{+}\rightarrow\mathbb{C}^{-}$ via:
$$
G(\xi)=(\Phi\circ G_{a})(\xi.1)\ \ \  \xi\in\mathbb{C}^{+}.
$$
Referring to the note after the Proposition \ref{pro:2}, it follows that there is a  probability measure $\mu=\mu_{a}\in \mathcal{B}_{fm}(\mathbb{R})$ on $\mathbb{R},$ which we call the  distribution of $a$, such that:
\begin{eqnarray}
\label{eq:13}
 G(\xi)=G_{\mu_{a}}(\xi)=\int_{\mathbb{R}}\frac{d\mu_{a}(t)}{\xi-t}\ \ \ \xi\in\mathbb{C}^{+}.
\end{eqnarray}
Note that the atoms of $\mu_{a}$ are determined via  Equation \eqref{eq:3}.

\section{Representation of  the Cauchy Transform Using Semicircular Random Variables}
$$$$
The proof of the following  results are based on the Theorem \ref{th:7} of Section 2. Indeed, for given $D$ and $\eta$ as in that theorem, we shall verify the conditions \eqref{eq:8} and \eqref{eq:9} for $b=\xi.1\ \ \xi\in\mathbb{C}^{+}$ and $G_{s}(b)$ a diagonal matrix of complex analytic functions. Then, Theorem \ref{th:7} of Section 2 will guarantee us that there is a semicircular random variable $s$ with first moment $D$ and variance $\eta$ so that $G_{s}(b)$ is the restriction of the operator-valued Cauchy transform of $s$ to $\mathbb{C}^{+}.1.$ Next, our variances $\eta$ will be explicitly constructed as $\eta(b)=v^{*}bv$, ~\cite[Theorem 4.1.]{VP} with $v$ obtained from the Jacobi coefficients of the given compactly supported probability measure.
\newline\\
The first  result deals with the finite dimensional matrix-valued representations of the  Cauchy transform. Here, we consider $\ell_{2}^{n}$ as $\mathbb{C}^n$ with its canonical orthonormal basis $\{e_k\}_{k=1}^{n}.$
\begin{proposition}
\label{pro:11}
Let $\mu$ be a probability measure with compact support in $\mathbb{R}.$ Then there exist two sequences $s_{n}^{(1)}$ and $s_{n}^{(2)}\ (n\geq1)$of self-adjoint operator valued semicircular random variables with associated operator-valued Cauchy transforms $G_{s_{n}^{(1)}}:\mathbb{H}^{+}(M_{n}(\mathbb{C}))\rightarrow \mathbb{H}^{-}(M_{n}(\mathbb{C}))$ and $G_{s_{n}^{(2)}}:\mathbb{H}^{+}(M_{n}(\mathbb{C}))\rightarrow \mathbb{H}^{-}(M_{n}(\mathbb{C}))\ (n\geq1)$ such that the Cauchy transform $G_{\mu}:\mathbb{C}^{+}\rightarrow\mathbb{C}^{-}$ is represented as:
\begin{eqnarray}
\label{eq:10}
&&G_{\mu}(\xi)=\lim_{n\rightarrow\infty}\langle G_{s_{n}^{(1)}}(\xi.1_n)e_{1},e_{1} \rangle_{\ell_{2}^{n}}\ \ \xi\in\mathbb{C}^{+},\\
\nonumber
\end{eqnarray}
and
\begin{eqnarray}
\label{eq:11}
&&G_{\mu}(\xi)=\lim_{n\rightarrow\infty}\langle G_{s_{n}^{(2)}}(\xi.1_n)e_{n},e_{n} \rangle_{\ell_{2}^{n}}\ \ \xi\in\mathbb{C}^{+}.\\
\nonumber
\end{eqnarray}
\end{proposition}
\begin{proof}
Let $$
G_{\mu}(\xi)=\frac{1}{\xi-\alpha_1-\displaystyle{\frac{\omega_1}{\xi-\alpha_2-\displaystyle{\frac{\omega_2}{{\xi-\alpha_3-}_{\displaystyle{\ddots}-\displaystyle{\frac{\omega_{n-1}}{\xi-\alpha_n-\displaystyle{\frac{\omega_n}{\xi-\alpha_{n+1}-}}_{\displaystyle{\ddots}}}}}}}}}}
$$
be the continued fraction representation of $G_{\mu}$ as in Theorem \ref{th:8}. To prove Equation \eqref{eq:10}, fix positive integer $n\geq1,$ then define $b=\xi.1_n,$ $D_{n}^{(1)}=(\alpha_{k}\delta_{kl})_{k,l=1}^{n}$ and the completely positive map $\eta_{n}^{(1)}$ via :
\begin{eqnarray}
&&\eta_{n}^{(1)}:M_{n}(\mathbb{C})\rightarrow M_{n}(\mathbb{C})\nonumber\\
&&\eta_{n}^{(1)}\big((a_{kl})_{k,l=1}^{n}\big)=(\omega_{k}^{\frac{1}{2}}\delta_{(k+1)l})_{k,l=1}^{n}(a_{kl})_{k,l=1}^{n}(\omega_{k-1}^{\frac{1}{2}}\delta_{k(l+1)})_{k,l=1}^{n}.\nonumber
\nonumber
\end{eqnarray}
Then, for the self-adjoint semicircular element $s_{n}^{(1)}$ with operator-valued Cauchy transform $G_{s_{n}^{(1)}}$ satisfying the functional equation \eqref{eq:8} of the form:
$$bG_{s_{n}^{(1)}}(b)=1+(D_{n}^{(1)}+\eta_{n}^{(1)}(G_{s_{n}^{(1)}}(b)))G_{s_{n}^{(1)}}(b),$$
we have $G_{s_{n}^{(1)}}(b)=(g_{n,n-k+1}(\xi)\delta_{kl})_{k,l=1}^{n}$ where :
$$
g_{n,n-k+1}(\xi)=\frac{1}{\xi-\alpha_1-\displaystyle{\frac{\omega_1}{\xi-\alpha_2-\displaystyle{\frac{\omega_2}{{\xi-\alpha_3-}_{\displaystyle{\ddots}-\displaystyle{\frac{\omega_{n-k+1}}{\xi-\alpha_{n-k+1}}}}}}}}}\ \ 1\leq k\leq n,\ \xi\in\mathbb{C}^{+},
$$
which can be identified as $(n-k+1)^{th}$ convergent of $G_{\mu}.$
Accordingly:
\begin{eqnarray}
&&G_{\mu}(\xi)=\lim_{n\rightarrow\infty}g_{n,n}(\xi)=\lim_{n\rightarrow\infty}\langle G_{s_{n}^{(1)}}(\xi.1_n)e_{1},e_{1} \rangle_{\ell_{2}^{n}}\ \ \xi\in\mathbb{C}^{+}.\nonumber
\nonumber
\end{eqnarray}
The proof of Equation \eqref{eq:11} is analogous by considering a fixed positive integer $n\geq1,$ then defining $b=\xi.1_n,$ $D_{n}^{(2)}=(\alpha_{n+1-k}\delta_{kl})_{k,l=1}^{n}$ and the completely positive map $\eta_{n}^{(2)}$ via :
\begin{eqnarray}
&&\eta_{n}^{(2)}:M_{n}(\mathbb{C})\rightarrow M_{n}(\mathbb{C})\nonumber\\
&&\eta_{n}^{(2)}\big((a_{kl})_{k,l=1}^{n}\big)=(\omega_{n-k+1}^{\frac{1}{2}}\delta_{k(l+1)})_{k,l=1}^{n}(a_{kl})_{k,l=1}^{n}(\omega_{n-k}^{\frac{1}{2}}\delta_{(k+1)l})_{k,l=1}^{n}.\nonumber
\nonumber
\end{eqnarray}
\end{proof}
The following theorem deals with the infinite dimensional matrix valued representation of the Cauchy transform. Here, we denote by $B(\ell_2(\mathbb{N}))$ the space of bounded operators on the separable Hilbert space $\ell_2(\mathbb{N}),$ and we  consider the orthonormal basis $\{e_n\}_{n=1}^{\infty}$ of $\ell_2(\mathbb{N})$ given by $e_n=\{\delta_{mn}\}_{m=1}^{\infty}\ (n\geq1).$
\begin{theorem}
\label{th: 12}
Let $\mu$ be a compactly supported probability measure in $\mathbb{R}$ with Jacobi coefficients $(\{\omega_n\}_{n=1}^{\infty},\{\alpha_n\}_{n=1}^{\infty})$. Then  there exist a self-adjoint $B(\ell_{2}(\mathbb{N}))-$valued semicircular random variable $s$ with first moment $D=(\alpha_{k}\delta_{kl})_{k,l=1}^{\infty}\in B(\ell_{2}(\mathbb{N})),$ and variance
\begin{eqnarray}
&&\eta:B(\ell_{2}(\mathbb{N}))\rightarrow B(\ell_{2}(\mathbb{N}))\nonumber\\
&&\eta((a_{kl})_{k,l=1}^{\infty})=(\omega_{k}^{\frac{1}{2}}\delta_{(k+1)l})_{k,l=1}^{\infty}(a_{kl})_{k,l=1}^{\infty}(\omega_{k-1}^{\frac{1}{2}}\delta_{k(l+1)})_{k,l=1}^{\infty},\nonumber
\nonumber
\end{eqnarray}
and an state $\rho:B(\ell_{2}(\mathbb{N}))\rightarrow\mathbb{C}$  such that the  Cauchy transform $G_{\mu}:\mathbb{C}^{+}\rightarrow\mathbb{C}^{-}$ is represented as:
$$
G_{\mu}(\xi)=(\rho\circ G_{s})(\xi.1)\ \ \xi\in\mathbb{C}^{+}.
$$
\end{theorem}
\begin{proof}
Let
$$
G_{\mu}(\xi)=\frac{1}{\xi-\alpha_1-\displaystyle{\frac{\omega_1}{\xi-\alpha_2-\displaystyle{\frac{\omega_2}{{\xi-\alpha_3-}_{\displaystyle{\ddots}-\displaystyle{\frac{\omega_{n-1}}{\xi-\alpha_n-\displaystyle{\frac{\omega_n}{\xi-\alpha_{n+1}-}}_{\displaystyle{\ddots}}}}}}}}}}
$$
be the continued fraction representation of $G_{\mu}$ as in Theorem \ref{th:8} and consider given $D$ and $\eta$ where using inequality \eqref{ineq:5} both of them are bounded in their corresponding norms. Now, as in proof of Proposition \ref{pro:11}, we observe that the diagonal matrix $$G_{s}(b)=(x_{kl}\delta_{kl})_{k,l=1}^{\infty},\ \ \ b=\xi.1\in B(\ell_{2}(\mathbb{N}))$$
with entries $x_{nn}$ of the form
$$
x_{nn}=\frac{1}{\xi-\alpha_n-\displaystyle{\frac{\omega_n}{\xi-\alpha_{n+1}-\displaystyle{\frac{\omega_{n+1}}{{\xi-\alpha_{n+2}}-\displaystyle{\frac{\omega_{n+2}}{\xi-\alpha_{n+3}-\displaystyle{\frac{\omega_{n+3}}{\xi-\alpha_{n+4}-}}_{\displaystyle{\ddots}}}}}}}}}
\ n\geq1,$$
satisfies the equation \eqref{eq:8} with condition \eqref{eq:9}. Consequently, for the state $\rho:B(\ell_{2}(\mathbb{N}))\rightarrow\mathbb{C}$ defined by:
$$
\rho(T)=\langle T(e_1),e_1\rangle_{\ell_{2}(\mathbb{N})}
$$
the assertion follows.
\end{proof}
\section{Atoms of Distributions of Matrix-Valued Semicircular Random Variables}
$$$$
In this section, the existence of atoms of distributions of finite dimensional matrix-valued semicircular random variables is discussed. First of all, we give a sufficient condition on the variance of a centered semicircular random variable so that its associated probability measure has atom.
\begin{theorem}
\label{th:9}
Let $G_{s}:\mathbb{H}^{+}(M_n(\mathbb{C}))\rightarrow\mathbb{H}^{-}(M_n(\mathbb{C}))$ be the operator valued Cauchy transform of a $M_{n}(\mathbb{C})$-valued semicircular random variable $s$ satisfying the functional equation \eqref{eq:8},  $b\in \mathbb{H}^{+}(M_n(\mathbb{C})), $ where  $D=0$  and $\eta:M_n(\mathbb{C})\rightarrow M_n(\mathbb{C})$ is a nilpotent completely positive map. Then the associated probability measure $\mu_{s}$ to $G_{s}$ has at least one atom .
\end{theorem}
\begin{proof}
Let $\eta,\cdots,\eta^{m-1}\neq0$ and $\eta^{m}=0,$ for some $m\geq1.$ Writing the functional equation in the form of $b-{G_{s}(b)}^{-1}=\eta(G_{s}(b))$ for $\|b^{-1}\|<<\infty$ it follows that:
$$\eta^{m-1}(b-{G_{s}(b)}^{-1})=\eta^{m}(G_{s}(b))=0\hspace{1 cm}Im(b)>0,\ \ \|b^{-1}\|<<\infty.$$
Now, if $\ker(\eta^{m-1})=\LARGE 0,$ then $G_{s}(b)=b^{-1}$ and using Equations \eqref{eq:3} and \eqref{eq:13} it follows that
$$\mu_{s}(\{0\})=1, $$
proving the assertion. Hence, we may assume $\ker(\eta^{m-1})\neq\LARGE 0.$ Pick $0\neq c\in \ker(\eta^{m-1})\cap M_{n}^{+}(\mathbb{C})$ with $\|c\|=1.$ Then by Schwarz inequality for completely positive maps ~\cite[p. 40]{VP}, it follows that:
$$\eta^{m-1}(c^{\frac{1}{2^n}})^{*}\eta^{m-1}(c^{\frac{1}{2^n}})\leq\|\eta^{m-1}(1)\|\eta^{m-1}(c^{\frac{1}{2^{n-1}}})\hspace{1 cm}n\geq1,$$
and by induction we conclude that $\eta^{m-1}(c^{\frac{1}{2^n}})=0\ (n\geq1).$ By defining: $$p:=^{s.o.t}\lim_{n\rightarrow\infty}c^{\frac{1}{2^n}},$$ it follows that $p$ is a projection in $\ker(\eta^{m-1}).$\\

Claim (1): There exists a unique projection $1\neq q\in\ker(\eta^{m-1})$ such that for any projection $p\in\ker(\eta^{m-1})$ we have: $p\leq q.$\\

We showed that there is at least one projection $p$ in  $\ker(\eta^{m-1}).$ Let $p_1,p_2$ be two projections in $\ker(\eta^{m-1}).$ Then :
$$\eta^{m-1}((p_1+p_2)^{\frac{1}{2^n}})^{*}\eta^{m-1}({(p_1+p_2)}^{\frac{1}{2^n}})\leq\|\eta^{m-1}(1)\|\eta^{m-1}((p_1+p_2)^{\frac{1}{2^{n-1}}})\ \ n\geq1,$$
and by induction it follows that $\eta^{m-1}((p_1+p_2)^{\frac{1}{2^n}})=0\ (n\geq1).$ Now, define $p_3:=^{s.o.t}\lim_{n\rightarrow\infty}(p_1+p_2)^{\frac{1}{2^n}},$ then $p_3$ is a projection in $\ker(\eta^{m-1}).$ On the other hand $$(p_1+p_2)^{\frac{1}{2^n}}\geq p_1,p_2\ \ \ (n\geq1),$$ yielding $p_3\geq p_1,p_2.$ Next, using the maximality argument and  by repeating this process there will be a unique maximal projection $q$ in $\ker(\eta^{m-1})$ such that for any other projection $p$ in it, we have $p\leq q.$ Finally, if $q=1,$ then using the same Schwarz inequality as above, and the canonical decomposition of elements of $M_{n}(\mathbb{C})$ into its positive elements it follows that $\eta^{m-1}=0,$ in contradiction to our hypothesis.\\

Claim(2): For any $0\neq c\in \ker(\eta^{m-1})\cap M_{n}^{+}(\mathbb{C}),$ we have $cq=qc=c.$\\

Indeed, since $c\in M_{n}^{+}(\mathbb{C}),$ by spectral theorem we have:
$$
c=\sum_{k=1}^{N}\lambda_{k}p_{k}
$$
where the projections $p_k$'s satisfy $\sum_{k=1}^{N}p_{k}=1,\ \  p_{k_1}p_{k_2}=0\ (1\leq k_1\neq k_2\leq N)$ and $\lambda_k\geq0\ (1\leq k\leq N).$ Now, define:
$$
r:=^{s.o.t}\lim_{n\rightarrow\infty}c^{\frac{1}{2^n}}.
$$
Then it follows that $r=\sum_{\lambda_{k}\neq0}p_k,$ yielding $p_k\leq r\ (\lambda_k\neq0).$ On the other hand, by definition of $q$ we have $r\leq q$ and hence $p_{k}\leq q\ (\lambda_k\neq0).$ But all of $p_{k}\leq q\ (\lambda_k\neq0),$ and $q$ are projections and, consequently, $p_{k}q=qp_{k}=p_{k}\ (\lambda_k\neq0).$ Now, by multiplying all sides by $\lambda_k\ (1\leq k\leq N),$ and taking summation the claim is proved.\\

Next, take $c_0=\frac{1}{i}({G_{s}(b)}^{-1}-b)$ where $b=iy.1\ (y>0).$ Then using the fact that $s$ is centered, $E_{B}(s^{2m-1})=0\ \
 m\geq1,$ by
\begin{eqnarray}
Re(G_{s}(b))&=&Re\Big(\sum_{m=0}^{\infty}b^{-1}E_{B}((sb^{-1})^m)\Big)=Re\Big(\sum_{m=0}^{\infty}i^{m+1}(-y)^{m+1}E_{B}(s^m)\Big)\nonumber\\
&=&\sum_{m=1}^{\infty}(-y^2)^{m}E_{B}(s^{2m-1})=0,\ \ \ y>\|s\|\nonumber
\nonumber
\end{eqnarray}
it follows that
$$
{G_{s}(b)}^{-1}-b=\Big( Re(G_{s}(b))+i.Im(G_{s}(b))\Big)^{-1}-b=i.\Big(-{Im(G_{s}(b))}^{-1}-\frac{b}{i}\Big),
$$
and, hence $c_0=\frac{1}{i}({G_{s}(b)}^{-1}-b)=Im(({G_{s}(b)}^{-1}-b))\geq 0.$ Now, by claim (2) for $c=c_0$ we have that: $$G_{s}(b)(1-q)=b^{-1}(1-q)=(1-q)G_{s}(b)\ \ \textrm{and}\ \ G_{s}(b)q=qG_{s}(b),$$ and hence:
\begin{eqnarray}
G_{s}(b)&=&(1-q)G_{s}(b)(1-q)+qG_{s}(b)q\nonumber\\
&=&b^{-1}(1-q)+qG_{s}(b)q\hspace{1 cm}b=iy.1\ (y>0), \|b^{-1}\|<<\infty.\nonumber
\nonumber
\end{eqnarray}
Now, applying  analytic continution for the complex function $tr_{n}\circ G_{s}|_{\mathbb{C}^{+}.1}:\mathbb{C}^{+}\rightarrow\mathbb{C}^{-}$ we conclude that:
$$
(tr_{n}\circ G_{s})(\xi.1)=tr_{n}(\xi^{-1}(1-q)+qG_{s}(\xi.1)q)\ \ \ \xi\in\mathbb{C}^{+},
$$
and, consequently, by another application of the Equations \eqref{eq:3} and \eqref{eq:13}:
$$ \mu_{s}(\{0\})=tr_{n}(1-q)>0,$$
completing the proof.
\end{proof}
Before stating a more concrete and special case of the above theorem, using M. D. Choi's representation of a completely positive map from a matrix algebra to another matrix algebra ~\cite{MC}, we have:
\begin{lemma}
Let $\eta$ be a completely positive map defined by
\begin{eqnarray}
&&\eta:M_n(\mathbb{C})\rightarrow M_n(\mathbb{C})\nonumber\\
&&\eta(a)=\sum_{j=1}^{n^2}a_jaa_{j}^{*},\ \ a_j\in M_n(\mathbb{C})\ (1\leq j\leq n^2).\nonumber
\nonumber
\end{eqnarray}
Then:
\newline\\
(i) if the map $\eta$ is nilpotent, then all matrices $a_{j}\ (1\leq j\leq n^2)$ are nilpotent,\\
(ii) if all matrices $a_{j}\ (1\leq j\leq n^2)$ are nilpotent and commute with each other, then the map $\eta$ is nilpotent.
\end{lemma}
\begin{proof}
(i) Let $\eta^{m}=0$ for some $m\geq1.$ Then, it follows that:
\begin{eqnarray}
\sum_{1\leq j_{1},\cdots,j_{m}\leq n^{2}} (a_{j_{1}}\cdots a_{j_{m}}a)(a_{j_{1}}\cdots a_{j_{m}}a)^{*}&=&\sum_{1\leq j_{1},\cdots,j_{m}\leq n^{2}} (a_{j_{1}}\cdots a_{j_{m}})aa^{*}(a_{j_{m}}^{*}\cdots a_{j_{1}}^{*})\nonumber\\
&=&\sum_{j_{1}=1}^{n^2}\cdots \sum_{j_{m}=1}^{n^2}(a_{j_{1}}\cdots a_{j_{m}})aa^{*}(a_{j_{m}}^{*}\cdots a_{j_{1}}^{*})\nonumber\\
&=&\sum_{j_{1}=1}^{n^2}a_{j_{1}}\Big(\cdots \Big(\sum_{j_{m}=1}^{n^2} a_{j_{m}}aa^{*}a_{j_{m}}^{*}\Big)\cdots \Big)a_{j_{1}}^{*}\nonumber\\
&=&\eta^{m}(aa^{*})\nonumber\\
&=&0\nonumber
\nonumber
\end{eqnarray}
for all $a\in M_n(\mathbb{C}).$ Consequently, by positivity of all elements of the form $(a_{j_{1}}\cdots a_{j_{m}}a)(a_{j_{1}}\cdots a_{j_{m}}a)^{*}$ it follows that $$(a_{j_{1}}\cdots a_{j_{m}}a)(a_{j_{1}}\cdots a_{j_{m}}a)^{*}=0\ \ (1\leq j_{1},\cdots,j_{m}\leq n^{2}),$$
for all $a\in M_n(\mathbb{C}).$ On the other hand ,  $M_n(\mathbb{C})$ is a $C^{*}$-algebra and hence:
$$a_{j_{1}}\cdots a_{j_{m}}a=0\ \ (1\leq j_{1},\cdots,j_{m}\leq n^{2}),$$ for all $a\in M_n(\mathbb{C}),$ or equivalently:
$$a_{j_{1}}\cdots a_{j_{m}}=0\ \ (1\leq j_{1},\cdots,j_{m}\leq n^{2}).$$ Now, take $j_{1}=\cdots=j_{m}=j$ where $1\leq j\leq n^2$ and the desired result is proved.\\

(ii) Since $a_{j}\ \ (1\leq j\leq n^2)$ are nilpotent, it follows that $a_{j}^{n}=0\ \ (1\leq j\leq n^2).$ Put $m=n^3,$ then by commutativity of the these matrices it follows that:
\begin{eqnarray}
a_{j_{1}}\cdots a_{j_{m}}&=&\prod_{p=0}^{n-1}\prod_{q=1}^{n^2}a_{j_{pn^2+q}}=a_{j_{(j_{1},\cdots,j_{m})}}^{n}\prod\prod_{j_{pn^2+q}\neq j_{(j_{1},\cdots,j_{m})}}a_{j_{pn^2+q}}=0,\nonumber
\nonumber
\end{eqnarray}
for all $1\leq j_{1},\cdots,j_{m}\leq n^{2}$. Consequently,
\begin{eqnarray}
\eta^{m}(a)&=&\sum_{1\leq j_{1},\cdots,j_{m}\leq n^{2}} (a_{j_{1}}\cdots a_{j_{m}})a(a_{j_{m}}^{*}\cdots a_{j_{1}}^{*})=0\nonumber
\nonumber
\end{eqnarray}
for all $a\in M_n(\mathbb{C}).$
\end{proof}
Considering above lemma, for a category of nilpotent $\eta's$ we have:
\begin{corollary}
Let $G_{s}:\mathbb{H}^{+}(M_n(\mathbb{C}))\rightarrow\mathbb{H}^{-}(M_n(\mathbb{C}))$ be the operator-valued Cauchy transform of a $M_{n}(\mathbb{C})$-valued semicircular random variable $s$ satisfying the functional equation \eqref{eq:8},  $b\in \mathbb{H}^{+}(M_n(\mathbb{C})),$ $D=0$ and the completely positive map $\eta$ is given by:
\begin{eqnarray}
&&\eta:M_n(\mathbb{C})\rightarrow M_n(\mathbb{C})\nonumber\\
&&\eta(a)=\sum_{j=1}^{n^2}a_jaa_{j}^{*},\ \ a_{j}^{n}=0\ (1\leq j\leq n^2),\ \ a_{j_1}a_{j_2}=a_{j_2}a_{j_1}\ (1\leq j_1,j_2\leq n^2).\nonumber
\nonumber
\end{eqnarray}
Then  the associated probability measure $\mu_{s}$ to $G_{s}$ has at least one atom at $x=0.$
\end{corollary}
A few examples of interest are discussed  in connection to the Theorem 4.1. :
\begin{remarks}
(i)  The converse of the assertion of the Theorem \ref{th:9}. does not hold. To see this, let $n=2$, $0\neq|\alpha|\neq|\beta|\neq0$ and define: $$a_1=\left(\begin{array}{cc}
0 & \alpha\\
\beta & 0
\end{array}
\right),\ \ a_j=0\ (2\leq j\leq 4).$$
Under these conditions, considering $D=0$ and $\eta(a)=a_{1}aa_{1}^{*}$ in the Equation \eqref{eq:8} with the restriction condition \eqref{eq:9}, and solving it for $G_{s}(b)$ yields:
$$
tr_{2} (G_{s}(b))=\frac{(|\alpha|^2+|\beta|^2)\Big(\xi^2+\sqrt{(-|\alpha|^2+|\beta|^2-\xi^2)^2-4.|\alpha|^2\xi^2}\Big)-(|\alpha|^2-|\beta|^2)^2}{4\xi|\alpha|^2|\beta|^2},
$$
and, by considering the Equations  \eqref{eq:3} and \eqref{eq:13} it follows that:
$$
\mu_{s}(\{r\})= 0\ \ \textrm{if}\ \ \ r\neq0,\ \frac{1}{2}\Big(1-\left|\frac{\beta}{\alpha}\right|^{2sgn(1-|\frac{\beta}{\alpha}|)}\Big)\ \ \ \textrm{if}\ \ r=0.
$$
(ii) If the assumption of nilpotency of the map $\eta$ is violated in the statement of the Theorem \ref{th:9}, then the associated distribution can have no atom. To see this,  let $n=2$, $|\alpha|=|\beta|\neq0$ and define: $$a_1=\left(\begin{array}{cc}
0 & \alpha\\
\beta & 0
\end{array}
\right),\ \ a_j=0\ (2\leq j\leq 4).$$
Under these conditions, considering $D=0$ and $\eta(a)=a_{1}aa_{1}^{*}$ in the Equation \eqref{eq:8} with the restriction condition \eqref{eq:9}, and solving it for $G_{s}(b)$ yields:
$$
tr_{2} (G_{s}(b))= \frac{\xi-\sqrt{\xi^2-4|\alpha|^2}}{2|\alpha|^2}
$$
Now, by considering the Equation \eqref{eq:13}, it follows that this is the Cauchy transform of the standard semicircular law of Wigner which has no atom.
\newline\\
(iii) In the  Theorem \ref{th:9} , for given nilpotent map  $\eta$ the associated probability measure to
the operator-valued Cauchy transform $G_{s}(b)$ may not be purely atomic. To see this, let $n=3$ and define :
$$
a_1 = \left(\begin{array}{ccc}
0&1&1\\
0&0&1\\
0&0&0
\end{array}\right),\ \ a_j=0\ (3\leq j\leq 9).
$$
Under these conditions, considering $D=0$ and  $\eta(a)=a_{1}aa_{1}^{*}$ in the Equation \eqref{eq:8} with the restriction condition \eqref{eq:9}, and solving it for $G_{s}(b)$ with the Groebner basis method ~\cite{DC - JL - DO} in Mathematica, it follows  that:
$$
tr_{3}(G_{s}(b))=\frac{1}{3}\Big(\frac{1}{\xi}+\frac{\xi-\sqrt{\xi^2-4}}{2}+(\xi^2-1)\big(\frac{\xi-\sqrt{\xi^2-4}}{2}\big)^3\Big),
$$
and, by considering the Equations  \eqref{eq:3} and \eqref{eq:13} it follows that:
$$
\mu_{s}(\{r\})=0\ \ \textrm{if}\ \  r\neq0,\ \frac{1}{3}\ \textrm{if}\ \ r=0,
$$
showing that $\mu_{s}$ is not purely atomic.
\end{remarks}
We end this section by mentioning a ``covering property" of distributions of matrix-valued semicircular random variables. Before that, we need a definition:
\begin{definition}
Let $\mu$ and $\nu$ be two probability measures on $\mathbb R$. We shall say that
$\nu$ is a component of $\mu$ if there exists a finite family $\{\nu_1,\dots,\nu_n\}$ of probability measures so that
$\nu\in\{\nu_1,\dots,\nu_n\}$ and $$\mu=\sum_{j=1}^n\alpha_j\nu_j$$ for some
$\alpha_1,\dots,\alpha_n\in[0,1]$ satisfying $\alpha_1+\cdots+\alpha_n=1$.
\end{definition}
\begin{proposition}
Any probability measure $\mu$ on $\mathbb{R}$ whose support is a finite set can be realized
as a component of a semicircular distribution $\mu_{s}$ of some matrix-valued semicircular random variable $s$ with nilpotent variance.
\end{proposition}
\begin{proof}
Assume $|Supp(\mu)|=n<\infty$, and let
$$
G_{\mu}(\xi)=\frac{1}{\xi-\alpha_1-\displaystyle{\frac{\omega_1}{\xi-\alpha_2-\displaystyle{\frac{\omega_2}{{\xi-\alpha_3-}_{\displaystyle{\ddots}\displaystyle{\frac{\omega_{n-1}}{\xi-\alpha_n}}}}}}}}
$$
be the continued fraction representation of $G_{\mu}$ as in Theorem \ref{th:8}. Then as in the proof of the Proposition \ref{pro:11}, define $b=\xi.1_n,\ D_n=(\alpha_{k}\delta_{kl})_{k,l=1}^{n}$ and the nilpotent completely positive map $\eta_n$ via:
\begin{eqnarray}
&&\eta_{n}:M_{n}(\mathbb{C})\rightarrow M_{n}(\mathbb{C})\nonumber\\
&&\eta_{n}\big((a_{kl})_{k,l=1}^{n}\big)=(\omega_{k}^{\frac{1}{2}}\delta_{(k+1)l})_{k,l=1}^{n}(a_{kl})_{k,l=1}^{n}(\omega_{k-1}^{\frac{1}{2}}\delta_{k(l+1)})_{k,l=1}^{n}.\nonumber
\nonumber
\end{eqnarray}
Then for the self-adjoint semicircular element $s=s_{n}$ with operator-valued Cauchy transform $G_{s}$ satisfying the functional equation \eqref{eq:8} in the form:
$$bG_{s}(b)=1+(D_{n}+\eta_{n}(G_{s}(b)))G_{s}(b),$$
we have $$G_{\mu}(\xi)=\langle G_{s}(\xi.1_n)e_1,e_1\rangle_{\ell_{2}^{n}}\ \ \  \xi\in\mathbb{C}^{+}.$$ Next, let $\mu_k\ \ (1\leq k\leq n)$ be a finitely supported probability measure on $\mathbb{R}$ with associated Cauchy transform:
$$
G_{\mu_{n-(k-1)}}(\xi)=\frac{1}{\xi-\alpha_1-\displaystyle{\frac{\omega_1}{\xi-\alpha_2-\displaystyle{\frac{\omega_2}{{\xi-\alpha_3-}_{\displaystyle{\ddots}\displaystyle{\frac{\omega_{k-1}}{\xi-\alpha_k}}}}}}}}\ \ (1\leq k\leq n).
$$
Note that $\mu=\mu_1.$ Then by proof of the Proposition \ref{pro:11}, we have
\begin{eqnarray}
\label{eq:14}
&&G_{\mu_{k}}(\xi)=\langle G_{s}(\xi.1_n)e_k,e_k\rangle_{\ell_{2}^{n}}  \ \ \ (1\leq k\leq n), \ \xi\in\mathbb{C}^{+}. \\
\nonumber
\end{eqnarray}
Consequently, by Equations \eqref{eq:13}  and \eqref{eq:14},  we have:
$$
G_{\mu_{s}}(\xi)=tr_{n}(G_{s}(\xi.1))=\frac{1}{n}\sum_{k=1}^{n}\langle G_{s}(\xi.1)e_k,e_k \rangle_{\ell_{2}^{n}}=\frac{1}{n}\sum_{k=1}^{n}G_{\mu_k}(\xi)  \ \ \ \xi\in\mathbb{C}^{+},
$$
and by the Equation \eqref{eq:3} it follows that:
$$
\mu_{s}=\frac{1}{n}\sum_{k=1}^{n}\mu_{k},
$$
proving the desired result.
\end{proof}


\section*{Acknowledgement}
									
The author would like to express his thanks to Prof. Serban Belinschi for his guidance and support on this paper.


\end{document}